\documentclass[11pt,reqno]{amsart}

\usepackage[T1]{fontenc}
\usepackage[utf8]{inputenc}
\usepackage[a4paper,margin=2.8cm]{geometry}
\usepackage{amsmath,amssymb,amsfonts,amsthm,mathtools}
\usepackage{enumitem}
\usepackage{microtype}
\usepackage[hidelinks,bookmarksdepth=3]{hyperref}
\usepackage{aliascnt}

\numberwithin{equation}{section}

\theoremstyle{plain}
\newtheorem{theorem}{Theorem}[section]

\newaliascnt{lemma}{theorem}
\newtheorem{lemma}[lemma]{Lemma}
\aliascntresetthe{lemma}

\newaliascnt{proposition}{theorem}
\newtheorem{proposition}[proposition]{Proposition}
\aliascntresetthe{proposition}

\newaliascnt{fact}{theorem}
\newtheorem{fact}[fact]{Fact}
\aliascntresetthe{fact}

\newaliascnt{claim}{theorem}
\newtheorem{claim}[claim]{Claim}
\aliascntresetthe{claim}

\newaliascnt{corollary}{theorem}
\newtheorem{corollary}[corollary]{Corollary}
\aliascntresetthe{corollary}

\theoremstyle{definition}
\newaliascnt{definition}{theorem}

\aliascntresetthe{definition}

\newaliascnt{remark}{theorem}

\aliascntresetthe{remark}
\usepackage{xcolor}
\usepackage[nameinlink,noabbrev]{cleveref}

\crefname{theorem}{Theorem}{Theorems}
\crefname{lemma}{Lemma}{Lemmas}
\crefname{claim}{Claim}{Claims}
\crefname{fact}{Fact}{Facts}
\crefname{proposition}{Proposition}{Propositions}
\crefname{corollary}{Corollary}{Corollaries}
\crefname{definition}{Definition}{Definitions}
\crefname{remark}{Remark}{Remarks}

\newcommand{\Prb}{\mathbb P}
\newcommand{\E}{\mathbb E}

\DeclareMathOperator{\ind}{ind}

\title[Embedding Induced Bounded Degree Graphs]{Embedding Induced Bounded Degree Graphs}

\author[G. Carenini]{Gaia Carenini}
\address{Trinity College Cambridge, Department of Pure Mathematics and Mathematical Statistics, Centre for Mathematical Sciences, Wilberforce Road, Cambridge CB3 0WA, United Kingdom}
\email{gc645@cam.ac.uk}
\date{\today}

\begin{document}
\begin{abstract}
We prove a sparse embedding theorem for induced embeddings of bounded-degree graphs.  The theorem applies to pairs $G\subseteq \Gamma$: the graph $G$ supplies the positive edges of the target graph, while the ambient graph $\Gamma$ supplies the induced constraints which must be avoided.  Its main feature is that these two types of constraints are kept separate throughout the embedding process.

As an application, we show that, for every fixed $\Delta,r\ge2$, there are constants $A,C>0$ such that every $n$-vertex graph $H$ with maximum degree at most $\Delta$ satisfies $r_{\ind}(H;r)\le C n^{\Delta+2}(\log n)^A$.
This improves the exponent in the polynomial bound of Conlon, Fox and Zhao for bounded-degree induced Ramsey numbers.  The proof combines the aforementioned embedding theorem with a sparse random transference argument, in which the random host is used only to certify robust deterministic hypotheses for every colour class.

\end{abstract}
\maketitle
\section{Introduction}
Sparse embedding theorems are now central in extremal and Ramsey combinatorics. Starting with the sparse regularity method of Kohayakawa and R\"odl~\cite{KR}, and continuing through the sparse embedding results of Conlon, Fox and Zhao~\cite{CFZ} and the sparse blow-up lemma of Allen, B\"ottcher, H\`an, Kohayakawa and Person~\cite{ABHKP}, a general principle has emerged: under suitable regularity and inheritance assumptions, sparse pairs behave enough like random bipartite graphs to support bounded-degree embeddings.

These results are aimed mainly at ordinary embeddings, where one must realise the edges of the target graph while keeping candidate sets large and well distributed. Induced embeddings add a second constraint: non-edges of the target must map to genuine non-edges of the ambient graph. This creates a difficulty already visible in a greedy embedding. Choosing the image of a vertex imposes positive adjacency constraints on future neighbours, a local effect controlled by sparse inheritance. But it may also delete candidates for many future non-neighbours, since any ambient edge to the chosen image would violate inducedness. These global deletions can accumulate over many steps and are not controlled by the usual sparse embedding machinery. Thus an induced embedding procedure must preserve enough positive structure for future neighbours while avoiding forbidden ambient edges to future non-neighbours.
  
Our main contribution is a deterministic embedding theorem that separates these two tasks: preserving the sparse neighbourhood structure needed for edges, and controlling the accumulated forbidden edges created by inducedness.
  We work with a pair of graphs
$G\subseteq \Gamma$
on the same vertex set.  The graph $G$ supplies the positive edges of the copy, while the ambient graph $\Gamma$ records the edges which must be avoided when two vertices of the target graph are non-adjacent.  Thus an acceptable embedding $\phi:V(H)\to V(\Gamma)$ is required to satisfy
$xy\in E(H)$ implies that $\phi(x)\phi(y)\in E(G)$, and
$xy\notin E(H)$ implies that $\phi(x)\phi(y)\notin E(\Gamma)$.
In particular, the copy is monochromatic inside $G$ but induced inside the larger graph $\Gamma$.

The deterministic theorem applies to bounded-degree graphs in sparse partite pairs $G\subseteq \Gamma$.  Its hypotheses divide naturally into two groups.  The graph $G$ is required to satisfy the usual lower-regularity and inheritance assumptions needed for a sparse bounded-degree embedding.  The ambient graph $\Gamma$ is required to satisfy pseudorandom conflict estimates, ensuring that the forbidden edges created by already embedded non-neighbours do not remove too many future candidates.  Once these hypotheses are verified, the conclusion is deterministic: every bounded-degree target graph with the prescribed partite structure has an acceptable embedding.

This separation is essential for Ramsey applications.  In an arbitrary colouring of a sparse host $\Gamma$, a single colour class $G$ must realise all edges of the desired copy.  However, the induced non-edge constraints are imposed by the whole host $\Gamma$, not by the chosen colour class alone.  Positive sparse regularity inside $G$ is therefore insufficient; one must also control the ambient edges of $\Gamma$ which may create forbidden conflicts.  The main theorem is designed precisely to handle this two-graph situation.

As an application of our embedding theorem, we improve the best known explicit polynomial bound for induced Ramsey numbers of bounded-degree graphs.  Recall that the induced Ramsey number $r_{\ind}(H;r)$ is the least integer $N$ for which there exists an $N$-vertex graph $R$ such that every $r$-colouring of $E(R)$ contains a monochromatic induced copy of $H$. The existence of induced Ramsey numbers was proved independently by Deuber~\cite{Deuber}, Erd\H{o}s, Hajnal and P\'osa~\cite{EHP}, and R\"odl~\cite{Rodl}.  For bounded-degree graphs, a conjecture often attributed to Trotter asserts that, for every fixed $\Delta$ and $r$, the induced Ramsey number of every $n$-vertex graph of maximum degree at most $\Delta$ is polynomial in $n$.  This was proved by \L uczak and R\"odl~\cite{LR}.  Subsequent work of Fox and Sudakov~\cite{FSinduced} and Conlon, Fox and Zhao~\cite{CFZ} gave substantially stronger quantitative bounds.  In particular, Conlon, Fox and Zhao proved that $r_{\ind}(H;r)\le C n^{2\Delta+8}$ for every $n$-vertex graph $H$ with $\Delta(H)\le \Delta$, where $C$ depends only on $\Delta$ and $r$. We prove the following improvement.

\begin{theorem}
\label{thm:main}
Fix $\Delta,r\ge2$.  There are constants $A,C>0$, depending only on
$\Delta$ and $r$, such that every $n$-vertex graph $H$ with
$\Delta(H)\le\Delta$ satisfies
$r_{\ind}(H;r)\le C n^{\Delta+2}(\log n)^A$.
\end{theorem}

\subsection{The embedding theorem}

We now describe the main mechanism behind the deterministic embedding theorem.  During the greedy embedding, each unembedded vertex $y\in V(H)$ has two associated candidate sets.  The first, denoted $\widehat A_t(y)$, is the positive candidate set: it consists of vertices satisfying all positive adjacency constraints in $G$ imposed by already embedded neighbours of $y$.  The second, denoted $C_t(y)\subseteq \widehat A_t(y)$, is the induced candidate set: it is obtained by deleting those vertices which would create a forbidden ambient edge in $\Gamma$ to the image of an already embedded non-neighbour of $y$.

The positive candidate sets are controlled by sparse regularity inheritance.  The new issue is to control the cumulative loss from the induced deletions.  We do this by introducing a potential function which measures, simultaneously over all unembedded vertices, the discrepancy between the positive and induced candidate scales:
$$
\Phi_t
=
\sum_y
\left(
\Lambda^{-a_t(y)}
\frac{|\widehat A_t(y)|}{|C_t(y)|}
\right)^\tau .
$$
Here $a_t(y)$ is the number of already embedded neighbours of $y$, and the factor $\Lambda^{-a_t(y)}$ normalises the expected positive candidate scale.  The conflict estimates for $\Gamma$ imply that, at each embedding step, the average increase of this potential is small.  It follows that the induced candidate sets remain comparable to the positive candidate sets throughout the greedy phase.

Once almost all vertices have been embedded, we leave a separated buffer.  The separation ensures that the remaining positive and induced constraints have a tractable dependence structure.  We then embed the buffer using a sampling argument together with Haxell's independent transversal theorem \cite{Haxell}.  This completes the deterministic acceptable embedding.

\subsection{The Ramsey application}

To prove Theorem~\ref{thm:main}, we verify the deterministic hypotheses inside a sparse random partite host.  We take a random $(\Delta+1)$-partite graph $\Gamma$ with edge probability
$p\asymp \frac{1}{n\log n}$
and parts of size
$M\asymp n^{\Delta+2}(\log n)^A$.
With high probability, this graph satisfies the required inheritance and conflict estimates.

We then consider an arbitrary $r$-colouring of $E(\Gamma)$.  A sparse colour-focusing lemma produces a monochromatic lower-regular $(\Delta+1)$-partite block.  Taking $G$ to be the corresponding colour class and retaining $\Gamma$ as the ambient graph, the deterministic embedding theorem gives a monochromatic copy of $H$ which is induced in $\Gamma$.  Since the host has $O(M)$ vertices in each part, this yields $r_{\ind}(H;r)\le C n^{\Delta+2}(\log n)^A$. 

\subsection{Related work}

The proof belongs to the sparse regularity and sparse blow-up tradition.  Sparse regularity originates in the work of Kohayakawa and R\"odl~\cite{KR}; sparse counting and embedding methods were developed further by Conlon, Fox and Zhao~\cite{CFZ}; and the inheritance phenomena used here are closely related to those appearing in the sparse blow-up lemma of Allen, B\"ottcher, H\`an, Kohayakawa and Person~\cite{ABHKP}.  The main novelty is that the graph providing the positive edges and the graph imposing the induced constraints are allowed to be different.  This leads to a conflict-stable sparse embedding framework for induced copies.

The Ramsey application continues the line of quantitative induced Ramsey bounds for bounded-degree graphs initiated by \L uczak and R\"odl~\cite{LR} and developed by Fox and Sudakov~\cite{FSinduced} and Conlon, Fox and Zhao~\cite{CFZ}.  Theorem~\ref{thm:main} lowers the exponent in the Conlon--Fox--Zhao bound from $2\Delta+8$ to $\Delta+2$, up to a polylogarithmic factor.

The result is also adjacent to induced size-Ramsey theory and induced embedding in sparse hosts.  Gir\~ao and Hurley~\cite{GiraoHurley} proved an induced analogue of the Friedman--Pippenger theorem for bounded-degree trees in sparse expanders.  Hunter and Sudakov~\cite{HunterSudakov} proved linear induced size-Ramsey bounds for bounded-degree trees and, more generally, bounded-treewidth graphs.  These results are complementary to ours: they concern linear size-Ramsey bounds for tree-like families, whereas the present application concerns vertex induced Ramsey bounds for arbitrary bounded-degree graphs.

\subsection{Organisation of the paper}

Section~\ref{sec:preliminaries} introduces the notation and auxiliary tools used throughout the paper.  Section~\ref{sec:embedding} proves the deterministic conflict-stable embedding theorem.  Section~\ref{sec:random-host} verifies its hypotheses in the sparse random partite host.  Section~\ref{sec:ramsey-application} deduces Theorem~\ref{thm:main}.

\section*{Acknowledgements}
The author is grateful to Jacob Fox for bringing this problem to her attention. The author is also
grateful to Imre Leader for his guidance and support. This work was supported
by the CB European PhD Studentship funded by Trinity College, Cambridge.

\section{Preliminaries}
\label{sec:preliminaries}

All graphs in the paper are finite and simple.  If $F$ is a graph and
$X,Y\subseteq V(F)$, we write $e_F(X,Y)$ for the number of edges with one
endpoint in $X$ and one endpoint in $Y$.  Whenever this notation is used,
the sets $X$ and $Y$ will be disjoint. All partite graphs in the paper are understood to have no edges inside parts. For a vertex $v$ and a set
$X$, write $N_F(v;X)=N_F(v)\cap X$. 
If $W$ is a set of vertices, write
$N_F(W;X)=X\cap \bigcap_{w\in W}N_F(w)$,
with the convention $N_F(\varnothing;X)=X$.

\subsection*{Partite induced embeddings}

We shall work with pairs of graphs $G\subseteq\Gamma$ on the same vertex set.
The graph $G$ carries the positive edges of the target graph, while
$\Gamma$ is the ambient graph in which the copy is required to be induced.
Thus an injective map $\phi:V(H)\to V(\Gamma)$ is \emph{acceptable} if, for all distinct $x,y\in V(H)$,
$xy\in E(H)$ implies that $\phi(x)\phi(y)\in E(G)$, and $xy\notin E(H)$ implies that $\phi(x)\phi(y)\notin E(\Gamma)$.

We use a partite formulation.  The host graph $\Gamma$ has parts
$V_1,\ldots,V_k$, and the target graph $H$ is equipped with a proper
assignment $\chi:V(H)\to[k]$.  An embedding $\phi$ is \emph{part-respecting} if
$\phi(x)\in V_{\chi(x)}$ for every $x\in V(H)$.
A set $B\subseteq V(H)$ is a \emph{separated buffer} if distinct vertices of $B$ have distance at least three in $H$.  
\subsection*{Sparse regularity}

Let $0<\varepsilon,d,p\le1$.  A disjoint pair $(X,Y)$ is
$(\varepsilon,d,p)$-\emph{lower-regular} in $G$ if every
$X'\subseteq X$, $Y'\subseteq Y$ with $|X'|\ge\varepsilon |X|$, and $|Y'|\ge\varepsilon |Y|$ satisfies
$e_G(X',Y')\ge (d-\varepsilon)p|X'||Y'|$.
We shall use the following elementary slicing fact.

\begin{fact}[Slicing]
\label{fact:slicing}
Let $0<\varepsilon,d,p,\rho\le1$, and suppose that $(X,Y)$ is
$(\varepsilon,d,p)$-lower-regular in $G$.  If $X'\subseteq X$ and
$Y'\subseteq Y$ satisfy
$|X'|\ge\rho |X|$ and $|Y'|\ge\rho |Y|$, then $(X',Y')$ is $(\varepsilon/\rho,d,p)$-lower-regular in $G$.
\end{fact}

\begin{proof}
If $X''\subseteq X'$ and $Y''\subseteq Y'$ have
$|X''|\ge(\varepsilon/\rho)|X'|$ and $|Y''|\ge(\varepsilon/\rho)|Y'|$, 
then $|X''|\ge\varepsilon |X|$ and $|Y''|\ge\varepsilon |Y|$.  The
lower-regularity of $(X,Y)$ applies.
\end{proof}

\subsection*{Candidate-scale sets and conflict-good hosts}

Let $\Gamma$ be a $k$-partite graph with parts $V_1,\ldots,V_k$, each of
order comparable to $M$, and let $\Delta\ge1$.  We write $K=p^\Delta M$.
Given $0<\theta\le1$, a set $S\subseteq V_i$ is called a
$(\Delta,\theta)$-\emph{candidate-scale set}, or simply a \emph{candidate-scale set}, with
witness $W$, if $W\subseteq V(\Gamma)\setminus V_i$, $|W|\le\Delta$, $S\subseteq N_\Gamma(W;V_i)$, and $|S|\ge \theta p^{|W|}|V_i|$.

We say that $\Gamma$ is
$(\Delta,\theta,\xi,C)$-\emph{conflict-good at scale $M$} if the following
two estimates hold. First, if
$S\subseteq N_\Gamma(W;V_i)$ and
$S'\subseteq N_\Gamma(W';V_j)$ are candidate-scale sets in distinct parts,
then
\begin{equation}
        e_\Gamma(S,S')\le C(p+\xi)|S||S'|.
        \label{eq:edge-load-good}
\end{equation}

Second, the following weighted overlap estimate holds. Let $x$ be a label
with part $i$, and let $Y$ be a finite set of labels equipped with a part
map $\chi:Y\to[k]$ \footnote{The labels should be thought of as unembedded vertices of the target graph, with $\chi$ recording their parts.}. Let $W_x$ and $W_y$, $y\in Y$, be witness sets
of size at most $\Delta$, and suppose that each vertex appears in at most
$\Delta$ of the sets $W_y$.  If
$S_x\subseteq N_\Gamma(W_x;V_i)$ and
$S_y\subseteq N_\Gamma(W_y;V_{\chi(y)})$ are candidate-scale sets, then,
for all non-negative weights $(\omega_y)_{y\in Y}$,
\begin{equation}
        \sum_{\substack{y\in Y\\ \chi(y)=i}}
        \omega_y
        \frac{|S_x\cap S_y|}{|S_x||S_y|}
        \le
        \xi \sum_{\substack{y\in Y\\ \chi(y)=i}}\omega_y .
        \label{eq:weighted-overlap-good}
\end{equation}

\subsection*{Envelope-positive inheritance}

The sparse inheritance input available for random graphs is naturally stated
for ambient neighbourhoods in $\Gamma$, whereas the greedy embedding uses
positive neighbourhoods in the subgraph $G$.  We therefore use an inheritance
hypothesis with an explicit regularity hierarchy.

Let $G\subseteq\Gamma$ be $k$-partite graphs with common parts
$V_1,\ldots,V_k$.  Let
$\boldsymbol\varepsilon=(\varepsilon_0,\ldots,\varepsilon_{2\Delta})$ and $\boldsymbol\eta=(\eta_0,\ldots,\eta_{2\Delta-1})$ be positive parameters.  We say that $G\subseteq\Gamma$ is
$(\Delta,\theta,\boldsymbol\varepsilon,\boldsymbol\eta,d,p)$-\emph{envelope-positively inheriting at
scale $M$} if the following hold whenever $0\le h<2\Delta$ and
$S_i\subseteq V_i$, $S_j\subseteq V_j$, and $S_\ell\subseteq V_\ell$ are
candidate-scale sets. First, if $i\ne j$ and $(S_i,S_j)$ is
$(\varepsilon_h,d,p)$-lower-regular in $G$, then all but at most
$\varepsilon_h |S_i|$ vertices $v\in S_i$ satisfy
\begin{equation}
        |N_G(v;S_j)|\ge \frac d2 p|S_j|.
        \label{eq:inherit-degree}
\end{equation}
Second, if $i\ne j$, $j\ne\ell$, and
$(S_j,S_\ell)$ is $(\varepsilon_h,d,p)$-lower-regular in $G$, then all but at most
$\varepsilon_h |S_i|$ vertices $v\in S_i$ are such that
$(N_\Gamma(v;S_j),S_\ell)$ is
$(\eta_h,d,p)$-lower-regular in $G$.
Third, if $i,j,\ell$ are distinct and
$(S_j,S_\ell)$ is $(\varepsilon_h,d,p)$-lower-regular in $G$, then all but at most
$\varepsilon_h |S_i|$ vertices $v\in S_i$ are such that
$(N_\Gamma(v;S_j),N_\Gamma(v;S_\ell))$ is
$(\eta_h,d,p)$-lower-regular in $G$.

\subsection*{Auxiliary tools}

We shall use the following theorem of Haxell.

\begin{theorem}[Haxell~\cite{Haxell}]
\label{thm:haxell}
Let $F$ be a graph whose vertex set is partitioned into sets
$L_1,\ldots,L_m$.  Suppose that for every $I\subseteq[m]$, the induced
subgraph
$F\left[\bigcup_{i\in I}L_i\right]
$
has domination number at least $2|I|-1$.  Then $F$ has an independent
transversal; that is, there are vertices $v_i\in L_i$, $i\in[m]$, such that
$\{v_1,\ldots,v_m\}$ is independent in $F$.
\end{theorem}

In particular, we are interested in the following consequence of the latter.

\begin{corollary}
\label{cor:haxell-degree}
Let $F$ be a graph whose vertex set is partitioned into sets
$L_1,\ldots,L_m$.  If $|L_i|\ge2D$ for every $i$, and
$\Delta(F)\le D-1$, then $F$ has an independent transversal.
\end{corollary}

\begin{proof}
For $I\subseteq[m]$, the subgraph induced by $\bigcup_{i\in I}L_i$ has at
least $2D|I|$ vertices and maximum degree at most $D-1$.  Hence any
dominating set has size at least $2|I|$, and in particular at least
$2|I|-1$.  The result follows from \cref{thm:haxell}.
\end{proof}

We shall use the following simple sampling estimate.  Its short proof is
included for completeness.

\begin{fact}[Large-deviation bound for independent uniform samples]
\label{fact:large-deviation}
Let $\Omega_1,\ldots,\Omega_m$ be finite sets, and for each $i$ let
$S_i$ be a uniformly random $s_i$-element subset of $\Omega_i$.  Assume
that the random sets $S_1,\ldots,S_m$ are chosen independently.  Let
$A_i\subseteq \Omega_i$, and define $X=\sum_{i=1}^m |S_i\cap A_i|$. Let $\mu=\E X$.  Then, for every real $T\ge e\mu$,
$$
        \Prb(X\ge T)\le \left(\frac{e\mu}{T}\right)^T .
$$
\end{fact}
\begin{proof}
 For $u\in A_i$, write
$q_u=s_i/|\Omega_i|$.  Then
$\mu=\sum_{i=1}^m \sum_{u\in A_i} q_u $. For an integer $h\ge1$, we first bound $\E\binom Xh$.  Let
$A=\bigsqcup_{i=1}^m A_i$, where the disjoint union remembers the index $i$.
For a set $R\subseteq A$, write $R_i=R\cap A_i$ and $r_i=|R_i|$.  Then
$$
        \Prb(R_i\subseteq S_i)
        =
        \frac{(s_i)_{r_i}}{(|\Omega_i|)_{r_i}}
        \le
        \left(\frac{s_i}{|\Omega_i|}\right)^{r_i},
$$
where $(a)_b=a(a-1)\cdots(a-b+1)$.  Since the samples $S_i$ are
independent, for every $R\subseteq A$ of size $h$,
$$
        \Prb(R\subseteq \bigsqcup_{i=1}^m S_i)
        \le
        \prod_{u\in R} q_u .
$$
Therefore
$$
        \E\binom Xh
        =
        \sum_{\substack{R\subseteq A\\ |R|=h}}
        \Prb(R\subseteq \bigsqcup_{i=1}^m S_i)
        \le
        \sum_{\substack{R\subseteq A\\ |R|=h}}
        \prod_{u\in R}q_u
        \le
        \frac1{h!}\left(\sum_{u\in A}q_u\right)^h
        =
        \frac{\mu^h}{h!}.
$$
If $X\ge h$, then $\binom Xh\ge1$.  Hence
$$
        \Prb(X\ge h)
        \le
        \E\binom Xh
        \le
        \frac{\mu^h}{h!}
        \le
        \left(\frac{e\mu}{h}\right)^h .
$$
Now let $T\ge e\mu$ be real and put $h=\lceil T\rceil$.  Then
$$
        \Prb(X\ge T)=\Prb(X\ge h)
        \le
        \left(\frac{e\mu}{h}\right)^h .
$$
The function $x\mapsto (e\mu/x)^x$ is decreasing for $x\ge \mu$, and
$h\ge T\ge e\mu\ge \mu$.  Therefore
$$
        \left(\frac{e\mu}{h}\right)^h
        \le
        \left(\frac{e\mu}{T}\right)^T ,
$$
which proves the claim.
\end{proof}

\section{Embedding theorem for bounded-degree graphs}\label{sec:embedding}
This section contains the deterministic core of the paper.  We prove an
embedding theorem for sparse partite pairs $G\subseteq \Gamma$, where $G$
is responsible for the positive adjacency constraints and $\Gamma$ is the
ambient graph in which all non-edges of the target must remain non-edges.  The
hypotheses are correspondingly divided into positive assumptions on $G$, in
the form of lower-regularity and inheritance, and conflict assumptions on
$\Gamma$, which control the deletions forced by inducedness.  After stating
the theorem, we prove it by a greedy algorithm which maintains positive
candidate sets and induced candidate sets separately; a potential function
prevents the latter from shrinking too much, and a final separated buffer is
embedded using Haxell's independent transversal theorem.
\begin{theorem}[Embedding theorem for bounded-degree graphs]
\label{thm:conflict-stable-embedding}
Fix $\Delta,k\ge 1$, $0<d\le 1$, and $\kappa,C\ge 1$.
There exist positive constants
$\theta,\xi_0,\lambda_0,K_0,
$
an integer $n_0$, and vectors
$\boldsymbol\varepsilon=(\varepsilon_0,\ldots,\varepsilon_{2\Delta})$, $\boldsymbol\eta=(\eta_0,\ldots,\eta_{2\Delta-1})$,
such that the following holds.

Let $G\subseteq \Gamma$ be $k$-partite graphs with parts
$V_1,\ldots,V_k$, where
$M\le |V_i|\le \kappa M$
for every $i\in[k]$.
Let $0<p\le 1$, and write
$K=p^\Delta M$,
and let $H$ be an $n$-vertex graph with $n\ge n_0$,
$\Delta(H)\le \Delta$, and a proper assignment
$\chi:V(H)\to[k]$ satisfying
$|\chi^{-1}(i)|\le |V_i|$
for every $i\in[k]$.
Suppose that $H$ has a separated buffer $B$. Assume also that
$$
  K\ge K_0\log(2n),\qquad
  (p+\xi)n\log(2n)\le \lambda_0,\qquad
  0<\xi\le \min\{\xi_0,p\},
$$
and that the following conditions hold.
\begin{enumerate}
  \item[\textup{(i)}]
  $\Gamma$ is $(\Delta,\theta,\xi,C)$-conflict-good at scale $M$.

  \item[\textup{(ii)}]
  Every cross-pair $(V_i,V_j)$, $i\ne j$, is
  $(\varepsilon_0,d,p)$-lower-regular in $G$.

  \item[\textup{(iii)}]
  $G\subseteq \Gamma$ is
  $(\Delta,\theta,\boldsymbol\varepsilon,\boldsymbol\eta,d,p)$-envelope-positively
  inheriting at scale $M$.
\end{enumerate}
Then there exists a part-respecting injective map
$\phi:V(H)\to V(\Gamma)$
which is an acceptable embedding.
\end{theorem}

\begin{proof}
We choose the constants in the following order.  First choose a constant
$U\ge 1$ sufficiently large in terms of $C,\Delta,\kappa$, and let
$U_a=U^a$ for $0\le a\le \Delta$.  Then choose
$\Lambda$ sufficiently large in terms of $U_\Delta,d,\Delta$.  Next choose
$\gamma>0$, and define 
$\tau=\gamma\log(2n)$ and
$\rho=e^{-3/\gamma}\Lambda^{-\Delta}$.
Choose a  small enough constant $\lambda_* >0$, depending only on
$U,d,\Delta$; in particular, we require $\lambda_*\le d/(2U)$.  Finally, choose
$\theta>0$ sufficiently small in terms of $\rho,d,\Delta$, and choose the vector $\boldsymbol\varepsilon$ so that
$0<\varepsilon_0\le\varepsilon_1\le\cdots\le\varepsilon_{2\Delta}$.
More precisely, for every $0\le h<2\Delta$ we require
\begin{equation}
        0<\varepsilon_h\ll \eta_h\ll \lambda_*\varepsilon_{h+1}
        \quad\text{and}\quad
        \varepsilon_{2\Delta}\ll \rho d/\Delta^2.
        \label{eq:epsilon-hierarchy}
\end{equation}
Thus, any pair which is $(\varepsilon_h,d,p)$-lower-regular is also
$(\varepsilon_{h'},d,p)$-lower-regular for every $h'\ge h$.
Then choose $\xi_0,\lambda_0>0$ sufficiently small, choose $n_0$ sufficiently large that $\gamma\log(2n_0)\ge1$, and choose $K_0$ sufficiently
large.  All constants depend only on the fixed parameters of the theorem.  We let $\varepsilon_*=\varepsilon_{2\Delta}$.

Fix a separated buffer $B$, and order the vertices of
$V(H)\setminus B$ as
$x_1,\ldots,x_m$.
We embed these vertices greedily and embed the buffer vertices only at the end.
At time $t$, write $D_t=\{x_1,\ldots,x_t\}$.
For a partial embedding $\phi_t:D_t\to V(\Gamma)$ and an unembedded vertex
$y$, let $a_t(y)$ be the number of already embedded neighbours of $y$.
The raw positive candidate set is
$\widehat A_t(y)
        =
        N_G\bigl(\phi_t(N_H(y)\cap D_t);V_{\chi(y)}\bigr)$.
The raw induced candidate set $\widehat C_t(y)$ is obtained from
$\widehat A_t(y)$ by deleting every vertex adjacent in $\Gamma$ to the
image of an already embedded non-neighbour of $y$.  The available induced
candidate set is
$C_t(y)=\widehat C_t(y)\setminus \operatorname{Im}(\phi_t)$.
The embedding will always choose the next image from $C_t(y)$.  The raw sets
$\widehat A_t(y)$, rather than the available sets, are the ones used for
positive regularity.

Let
$\delta=d/4$.
We call $\phi_t$ \emph{positive-good} if it is part-respecting and injective, maps
every edge inside $D_t$ to an edge of $G$, and the following two conditions
hold. First, every unembedded vertex $y$ satisfies
\begin{equation}
        \delta^{a_t(y)}p^{a_t(y)}|V_{\chi(y)}|
        \le |\widehat A_t(y)|
        \le U_{a_t(y)}p^{a_t(y)}|V_{\chi(y)}|.
        \label{eq:positive-size}
\end{equation}
Second, for every unembedded edge $yz\in E(H)$, the pair
$(\widehat A_t(y),\widehat A_t(z))
$
is $(\varepsilon_{a_t(y)+a_t(z)},d,p)$-lower-regular in $G$.

The empty embedding is positive-good, since $U_0=1$.  Moreover, if
$\phi_t$ is positive-good, then each $\widehat A_t(y)$ is a
candidate-scale set, provided $\theta\le\delta^\Delta$.  Indeed,
$\widehat A_t(y)\subseteq
N_\Gamma(\phi_t(N_H(y)\cap D_t);V_{\chi(y)})$, the witness has size
$a_t(y)\le\Delta$, and the lower bound in \eqref{eq:positive-size} gives the
required size condition.

The role of positive-goodness is to isolate the part of the argument which is
identical in spirit to a sparse blow-up lemma: the raw positive candidate sets
have the expected scale, and all future positive candidate pairs remain
lower-regular.  The induced constraints will be added only through the sets
$C_t(y)$ and through the potential $\Phi_t$.

We first consider the purely positive extension step.

\begin{claim}[Positive inheritance exceptions]
\label{claim:positive-greedy}
If $\phi_t$ is positive-good and $x=x_{t+1}$, then there is a set
$B_t^+(x)\subseteq \widehat A_t(x)$ with
$|B_t^+(x)|\le C_2\Delta^2\varepsilon_* |\widehat A_t(x)|$, where $C_2$ depends only on $\Delta$, such that every unused vertex
$v\in\widehat A_t(x)\setminus B_t^+(x)$ which additionally satisfies
\begin{equation}
        |N_\Gamma(v;\widehat A_t(y))|
        \le Up|\widehat A_t(y)|
        \label{eq:upper-choice}
\end{equation}
for every future neighbour $y$ of $x$, extends $\phi_t$ to a
positive-good embedding of $D_t\cup\{x\}$.
\end{claim}

\begin{proof}
We list the possible positive failures.  Let $y$ be a future neighbour of
$x$.  Put $h=a_t(x)+a_t(y)$.  Since
$(\widehat A_t(x),\widehat A_t(y))
$
is $(\varepsilon_h,d,p)$-lower-regular and both sets are candidate-scale,
the envelope-positive inheritance condition gives that all but at most
$\varepsilon_h|\widehat A_t(x)|\le \varepsilon_*|\widehat A_t(x)|$ choices
$v\in\widehat A_t(x)$ satisfy
\begin{equation}
        |N_G(v;\widehat A_t(y))|
        \ge \frac d2 p|\widehat A_t(y)|.
        \label{eq:positive-degree}
\end{equation}
For such $v$, the new raw positive candidate set of $y$ has size at least
$$
        \frac d2 p|\widehat A_t(y)|
        \ge
        \frac d2 \delta^{a_t(y)}p^{a_t(y)+1}|V_{\chi(y)}|
        \ge
        \delta^{a_t(y)+1}p^{a_t(y)+1}|V_{\chi(y)}|,
$$
because $\delta=d/4$.

For the upper bound, suppose in addition that $v$ satisfies
\eqref{eq:upper-choice} for this future neighbour $y$.  Since
$\widehat A_{t+1}(y)=N_G(v;\widehat A_t(y))
        \subseteq N_\Gamma(v;\widehat A_t(y))$, 
we have
$$
        |\widehat A_{t+1}(y)|
        \le Up|\widehat A_t(y)|
        \le U_{a_t(y)+1}p^{a_t(y)+1}|V_{\chi(y)}|.
$$
This is the required upper bound in \eqref{eq:positive-size}.  If $y$ is not
a future neighbour of $x$, then $\widehat A_t(y)$ is unchanged.

It remains to preserve lower-regularity of future raw candidate pairs.  If a
future edge $yz$ is unaffected by the embedding of $x$, there is nothing to
prove.  Suppose first that exactly one of $y,z$, say $y$, is a future
neighbour of $x$.  Put $h=a_t(y)+a_t(z)$.  The one-sided part of
envelope-positive inheritance, applied to $\widehat A_t(x)$,
$\widehat A_t(y)$, and $\widehat A_t(z)$, gives
$(\eta_h,d,p)$-lower-regularity of
$(N_\Gamma(v;\widehat A_t(y)),\widehat A_t(z))$ for all but at most
$\varepsilon_h|\widehat A_t(x)|\le \varepsilon_*|\widehat A_t(x)|$
choices of $v$.
For choices of $v$ satisfying \eqref{eq:positive-degree} and
\eqref{eq:upper-choice}, the set $N_G(v;\widehat A_t(y))$ has size at least
$(d/2)p|\widehat A_t(y)|$, while
$|N_\Gamma(v;\widehat A_t(y))|
        \le Up|\widehat A_t(y)|$.
Hence $N_G(v;\widehat A_t(y))$ occupies relative density at least
$d/(2U)\ge\lambda_*$ inside $N_\Gamma(v;\widehat A_t(y))$.  By slicing and
\eqref{eq:epsilon-hierarchy},
$(N_G(v;\widehat A_t(y)),\widehat A_t(z))$ is
$(\varepsilon_{h+1},d,p)$-lower-regular, as required for the new rank
$a_{t+1}(y)+a_{t+1}(z)=h+1$.  This index is within the hierarchy: in the
one-sided case the future edge $yz$ and the additional neighbour $x$ give
$a_t(y)+1\le\Delta$ and $a_t(z)\le\Delta$, so $h+1\le2\Delta$.

If both $y$ and $z$ are future neighbours of $x$, let
$h=a_t(y)+a_t(z)$.  Since $yz$ is a future edge and both $y$ and $z$
are neighbours of $x$, the vertices $x,y,z$ span a triangle in $H$.
Thus their parts are pairwise distinct, because $\chi$ is proper.

The two-sided part of envelope-positive inheritance first
gives $(\eta_h,d,p)$-lower-regularity of the pair
$(N_\Gamma(v;\widehat A_t(y)),N_\Gamma(v;\widehat A_t(z)))$,
outside an exceptional set of size at most $\varepsilon_h|\widehat A_t(x)|$.
Applying
the same fixed-fraction and slicing argument in both coordinates gives
$(\varepsilon_{h+1},d,p)$-lower-regularity of the pair
$(N_G(v;\widehat A_t(y)),N_G(v;\widehat A_t(z)))$.
Since the hierarchy is increasing, this is also
$(\varepsilon_{h+2},d,p)$-lower-regular, as required for the new rank
$a_{t+1}(y)+a_{t+1}(z)=h+2$.  The index lies in the hierarchy, since before
embedding $x$, both $y$ and $z$ still have $x$ as a future neighbour.
Hence $a_t(y),a_t(z)\le\Delta-1$, and so $h+2\le2\Delta$.

There are only $O_\Delta(1)$ affected future pairs.  Let $B_t^+(x)$ be the
union of all positive exceptional sets.  Then, for a suitable constant
$C_2=C_2(\Delta)$,
$|B_t^+(x)|
        \le C_2\Delta^2\varepsilon_*|\widehat A_t(x)|$.
Every unused $v\in\widehat A_t(x)\setminus B_t^+(x)$ satisfying
\eqref{eq:upper-choice} for every future neighbour of $x$ is part-respecting,
preserves all raw candidate-size bounds, preserves all lower-regularity
conditions for future raw candidate pairs, and maps every newly completed edge
to an edge of $G$.  Hence it extends $\phi_t$ to a positive-good embedding.
\end{proof}

We now add the induced and availability constraints.  For an unembedded vertex
$y$, let
$$
        R_t(y)=
        \Lambda^{-a_t(y)}
        \frac{|\widehat A_t(y)|}{|C_t(y)|},
\quad\text{and}\quad 
\Phi_t=\sum_{y\notin D_t} R_t(y)^\tau.$$ Initially $\Phi_0\le n$.  We maintain the stronger bound
$\Phi_t\le n^2$. 
If $\Phi_t\le n^2$, then for every unembedded $y$,
$R_t(y)\le n^{2/\tau}\le e^{3/\gamma}$. Since $a_t(y)\le\Delta$, this implies
\begin{equation}
        |C_t(y)|
        \ge
        \rho |\widehat A_t(y)|,
        \qquad
        \rho=e^{-3/\gamma}\Lambda^{-\Delta}.
        \label{eq:C-comparable-A}
\end{equation}
Together with \eqref{eq:positive-size}, this gives
\begin{equation}
        |C_t(y)|
        \ge
        c\,p^{a_t(y)}|V_{\chi(y)}|,
        \qquad
        c=\rho\delta^\Delta.
        \label{eq:C-candidate-scale}
\end{equation}
In particular, if $\theta\le c/6$, then every set $C_t(y)$ and every subset
of $C_t(y)$ of size at least $|C_t(y)|/6$ is candidate-scale, with the same
witness as $\widehat A_t(y)$.

\begin{claim}[Induced greedy step]
\label{claim:induced-greedy-step}
Suppose that $\phi_t$ is an induced partial embedding, is positive-good, and
satisfies $\Phi_t\le n^2$.  Then, for $x=x_{t+1}$, there is a vertex
$v\in C_t(x)$ such that
$\phi_{t+1}=\phi_t\cup\{x\mapsto v\}$
is an induced partial embedding, is positive-good, and satisfies
\begin{equation}
        \Phi_{t+1}
        \le
        \bigl(1+C_3(p+\xi)\log(2n)\bigr)\Phi_t ,
        \label{eq:potential-step}
\end{equation}
where $C_3$ depends only on the fixed parameters.
\end{claim}

\begin{proof}
Let $x=x_{t+1}$.  By \eqref{eq:C-comparable-A},
$|C_t(x)|\ge \rho|\widehat A_t(x)|$. Let $B_t^+(x)$ be the positive exceptional set from \cref{claim:positive-greedy}.  Since
$|B_t^+(x)|\le C_2\Delta^2\varepsilon_*|\widehat A_t(x)|
$
and $\varepsilon_*\ll\rho/\Delta^2$, we have
$|C_t(x)\setminus B_t^+(x)|\ge \frac34 |C_t(x)|$.

Let $y$ be a future neighbour of
$x$.  Since $\phi_t$ is positive-good, the pair
$(\widehat A_t(x),\widehat A_t(y))
$
is $(\varepsilon_h,d,p)$-lower-regular for
$h=a_t(x)+a_t(y)$, and hence is $(\varepsilon_*,d,p)$-lower-regular.
The sets $C_t(x)$ and $C_t(y)$ have relative size at least $\rho$
inside $\widehat A_t(x)$ and $\widehat A_t(y)$, respectively.  By
\cref{fact:slicing}, $(C_t(x),C_t(y))$ is
$(\varepsilon_*/\rho,d,p)$-lower-regular.  If more than
$(\varepsilon_*/\rho)|C_t(x)|$ vertices $v\in C_t(x)$ failed
\begin{equation}\label{eq: ast}
  |N_G(v;C_t(y))|\ge \frac d4 p|C_t(y)|,
\end{equation}
then those vertices together with $C_t(y)$ would contradict
$(\varepsilon_*/\rho,d,p)$-lower-regularity, provided
$\varepsilon_*/\rho\ll d$.  Hence all but at most
$(\varepsilon_*/\rho)|C_t(x)|$ vertices of $C_t(x)$ satisfy (\ref{eq: ast}).

We also impose the upper-choice condition needed in
\cref{claim:positive-greedy}.  Since $C_t(x)$ and $\widehat A_t(y)$ are
candidate-scale sets in distinct parts, the edge-load estimate gives
$$
        \frac1{|C_t(x)|}
        \sum_{v\in C_t(x)} |N_\Gamma(v;\widehat A_t(y))|
        \le C(p+\xi)|\widehat A_t(y)|
        \le 2Cp|\widehat A_t(y)|,
$$
where we used $\xi\le p$.  By Markov's inequality, the proportion of
vertices $v\in C_t(x)$ for which
$|N_\Gamma(v;\widehat A_t(y))|>Up|\widehat A_t(y)|
$
is at most $2C/U$.  We choose $U$ so large that the total loss over the at
most $\Delta$ future neighbours of $x$ is at most $|C_t(x)|/16$.

There are at most $\Delta$ future neighbours of $x$.  Therefore the set
$\mathcal P_t(x)$ of vertices in $C_t(x)\setminus B_t^+(x)$ satisfying
\begin{equation}
        |N_G(v;C_t(y))|
        \ge \frac d4 p|C_t(y)|
        \label{eq:future-neighbour-good}
\end{equation}
and
\begin{equation}
        |N_\Gamma(v;\widehat A_t(y))|
        \le Up|\widehat A_t(y)|
        \label{eq:future-neighbour-upper-good}
\end{equation}
for every future neighbour $y$ of $x$ has size at least
$|\mathcal P_t(x)|\ge \frac13 |C_t(x)|$. In particular, $\mathcal P_t(x)$ is candidate-scale, provided
$\theta\le c/6$.

Fix $v\in\mathcal P_t(x)$.  We first consider a future neighbour $y$ of
$x$.  Then
$\widehat A_{t+1}(y)=N_G(v;\widehat A_t(y))
$
and
$C_{t+1}(y)=N_G(v;C_t(y))$.  No additional injectivity deletion is needed here,
since $y$ is adjacent to $x$, hence $\chi(y)\ne\chi(x)$ and
$v\notin C_t(y)$.
The positive-good upper bound gives
$|\widehat A_{t+1}(y)|
        \le
        U_{a_t(y)+1}p^{a_t(y)+1}|V_{\chi(y)}|$,
while \eqref{eq:positive-size} gives $|\widehat A_t(y)|
        \ge
        \delta^{a_t(y)}p^{a_t(y)}|V_{\chi(y)}|$. 
Together with \eqref{eq:future-neighbour-good}, this yields
$$
\begin{aligned}
        \frac{R_{t+1}(y)}{R_t(y)}
        &=
        \Lambda^{-1}
        \frac{|\widehat A_{t+1}(y)|}{|\widehat A_t(y)|}
        \frac{|C_t(y)|}{|C_{t+1}(y)|}        \\
        &\le
        \Lambda^{-1}
        \cdot
        \frac{U_{a_t(y)+1}p^{a_t(y)+1}|V_{\chi(y)}|}
             {\delta^{a_t(y)}p^{a_t(y)}|V_{\chi(y)}|}
        \cdot
        \frac{4}{dp}        \\
        &\le
        \Lambda^{-1}\frac{4U_\Delta}{d\delta^\Delta}.
\end{aligned}
$$
Thus, by choosing
$\Lambda\ge \frac{4U_\Delta}{d\delta^\Delta}$,
we obtain
$R_{t+1}(y)\le R_t(y)
$
for every future neighbour $y$ of $x$.

Now let $y$ be an unembedded non-neighbour of $x$.  Then no new positive
constraint is imposed on $y$, so
$\widehat A_{t+1}(y)=\widehat A_t(y)$. The only possible increase in $R_t(y)$ comes from deleting from $C_t(y)$ the
ambient neighbours of $v$, and, if $v\in C_t(y)$, from deleting $v$
itself for injectivity.  Define
$$
        \ell_y(v)=
        \frac{|N_\Gamma(v;C_t(y))|+\mathbf 1_{v\in C_t(y)}}{|C_t(y)|}.
$$
Then
$|C_{t+1}(y)|
        \ge
        (1-\ell_y(v))|C_t(y)|.
$
Hence, whenever $\ell_y(v)\le1/2$,
$R_{t+1}(y)\le R_t(y)(1+2\ell_y(v))$.

We next discard the vertices for which some individual loss term is too large.
Let $\mathcal B_t^0(x)$ be the set of vertices $v\in\mathcal P_t(x)$ for
which
$\ell_y(v)>\frac1{2\tau}
$
for at least one future non-neighbour $y$ of $x$. Fix such a future non-neighbour $y$.  If $\chi(y)\ne\chi(x)$, then
$\mathcal P_t(x)$ and $C_t(y)$ are candidate-scale sets in distinct parts,
and the edge-load estimate gives
$$
        \frac1{|\mathcal P_t(x)|}
        \sum_{v\in\mathcal P_t(x)}
        \frac{|N_\Gamma(v;C_t(y))|}{|C_t(y)|}
        \le C(p+\xi).
$$
There is no injectivity contribution in this case, because the two sets lie in
distinct parts.  Hence the fraction of $v\in\mathcal P_t(x)$ with
$\ell_y(v)>1/(2\tau)$ is at most $2C\tau(p+\xi)$.
If $\chi(y)=\chi(x)$, then there are no ambient edges inside the part, and
$\ell_y(v)\le 1/|C_t(y)|\le 1/(cK)$.  Taking $K_0$ sufficiently large gives
$1/(cK)\le 1/(2\tau)$, so no vertex is bad for this $y$.

Taking a union bound over at most $n$ possible vertices $y$, and using
$(p+\xi)n\log(2n)\le\lambda_0$, we obtain
$|\mathcal B_t^0(x)|\le \frac12|\mathcal P_t(x)|
$
after choosing $\lambda_0$ sufficiently small.  Let
$\mathcal P_t^0(x)=\mathcal P_t(x)\setminus\mathcal B_t^0(x)$.
Then $|\mathcal P_t^0(x)|\ge |C_t(x)|/6$, so $\mathcal P_t^0(x)$ is
candidate-scale.  Moreover,
\begin{equation}
        \ell_y(v)\le \frac1{2\tau}
        \label{eq:ell-small}
\end{equation}
for every $v\in\mathcal P_t^0(x)$ and every future non-neighbour $y$ of
$x$.

We average the weighted loss over $v\in\mathcal P_t^0(x)$.  For the
ambient-edge part, only vertices $y$ with $\chi(y)\ne\chi(x)$ contribute.
Since $\mathcal P_t^0(x)$ and $C_t(y)$ are candidate-scale sets in distinct
parts, \eqref{eq:edge-load-good} gives
$$
        \frac1{|\mathcal P_t^0(x)|}
        \sum_{v\in\mathcal P_t^0(x)}
        \frac{|N_\Gamma(v;C_t(y))|}{|C_t(y)|}
        =
        \frac{e_\Gamma(\mathcal P_t^0(x),C_t(y))}
             {|\mathcal P_t^0(x)||C_t(y)|}
        \le C'(p+\xi).
$$
For the injectivity part, only vertices $y$ with $\chi(y)=\chi(x)$
contribute.  The witness sets defining $C_t(y)$ have bounded incidence:
indeed each already embedded vertex is a neighbour of at most $\Delta$
unembedded vertices.  Hence the weighted overlap estimate
\eqref{eq:weighted-overlap-good}, applied with weights
$\omega_y=R_t(y)^\tau$, gives
$$
        \sum_{\substack{y\notin D_t\cup\{x\}\\ xy\notin E(H)\\
                        \chi(y)=\chi(x)}}
        R_t(y)^\tau
        \frac{|\mathcal P_t^0(x)\cap C_t(y)|}
             {|\mathcal P_t^0(x)||C_t(y)|}
        \le
        \xi
        \sum_{\substack{y\notin D_t\cup\{x\}\\ xy\notin E(H)}}
        R_t(y)^\tau .
$$
Combining the two estimates, we obtain
\begin{equation}
        \frac1{|\mathcal P_t^0(x)|}
        \sum_{v\in\mathcal P_t^0(x)}
        \sum_{\substack{y\notin D_t\cup\{x\}\\ xy\notin E(H)}}
        R_t(y)^\tau\ell_y(v)
        \le C_3'(p+\xi)\Phi_t .
        \label{eq:average-loss}
\end{equation}
Thus some $v\in\mathcal P_t^0(x)$ satisfies the corresponding pointwise bound.

Choose such a vertex $v$.  By \eqref{eq:ell-small} and the elementary
inequality
$(1+2u)^\tau\le 1+4\tau u$
for $0\le u\le \frac1{2\tau}$,
the non-neighbour contribution to the potential increases by at most
$4\tau C_3'(p+\xi)\Phi_t$.
The neighbour contribution does not increase, and the vertex $x$ disappears
from the potential sum.  Hence
$\Phi_{t+1}
        \le
        \bigl(1+C_3(p+\xi)\log(2n)\bigr)\Phi_t
$
after increasing $C_3$.  Since $v\in C_t(x)$, all induced constraints with
already embedded vertices are satisfied.  Since $v\notin B_t^+(x)$, the
extension is positive-good.  This proves the claim.
\end{proof}

Starting from $\Phi_0\le n$, \cref{claim:induced-greedy-step} and the
hypothesis
$(p+\xi)n\log(2n)\le\lambda_0
$
give
$$
        \Phi_t
        \le
        n\exp\bigl(C_3(p+\xi)n\log(2n)\bigr)
        \le n^2
$$
throughout the greedy phase, after choosing $\lambda_0$ sufficiently small.
Thus all non-buffer vertices are embedded.

At the end of the greedy phase, every buffer vertex $b\in B$ has a final list
$L_b=C_m(b)$. Since $B$ is separated, no two vertices of $B$ are adjacent, and all
neighbours of $b$ have already been embedded.  By
\eqref{eq:C-candidate-scale},
$|L_b|\ge cK$. It remains to embed the buffer vertices.  For $b\in B$ and $u\in L_b$, put
\begin{equation}
        \ell_b(u)=
        \sum_{b'\ne b}\frac{\mathbf 1_{u\in L_{b'}}}{|L_{b'}|}
        +
        \sum_{b'\ne b}\frac{|N_\Gamma(u;L_{b'})|}{|L_{b'}|}.
        \label{eq:buffer-load}
\end{equation}
Each list $L_b$ is candidate-scale.  Moreover, since $B$ is separated, no
already embedded vertex is a neighbour of two different buffer vertices.
Therefore the witness sets defining the lists $L_b$ have bounded incidence,
indeed incidence at most one.

We estimate the average of $\ell_b$ over $L_b$.  The first term is an
injectivity-overlap term.  The weighted overlap estimate gives
$$
        \frac1{|L_b|}
        \sum_{u\in L_b}
        \sum_{b'\ne b}\frac{\mathbf 1_{u\in L_{b'}}}{|L_{b'}|}
        =
        \sum_{b'\ne b}
        \frac{|L_b\cap L_{b'}|}{|L_b||L_{b'}|}
        \le \xi |B|.
$$
The second term is an ambient-edge term.  For $b'$ in the same part as $b$,
there are no ambient edges between $L_b$ and $L_{b'}$.  For $b'$ in a
different part, the edge-load estimate gives
$$
        \frac{e_\Gamma(L_b,L_{b'})}{|L_b||L_{b'}|}
        \le C(p+\xi).
$$
Hence
\begin{equation}
        \frac1{|L_b|}\sum_{u\in L_b}\ell_b(u)
        \le C_4(p+\xi)|B|
        \le C_4(p+\xi)n .
        \label{eq:average-buffer-load}
\end{equation}
Put $\beta=C_4(p+\xi)n$.
By choosing $\lambda_0$ sufficiently small, we may assume $\beta$ is smaller
than a fixed absolute constant.  By Markov's inequality, the cleaned list
$L_b^*=\{u\in L_b:\ell_b(u)\le\sqrt\beta\}$
has size at least
$$
        |L_b^*|\ge (1-\sqrt\beta)|L_b|\ge \frac12 |L_b|\ge \frac{cK}{2}.
$$

Choose an even integer $D$ satisfying
$D_0\log(2n)\le D\le \frac{cK}{4}$, where $D_0$ will be chosen sufficiently large.  This is possible by taking
$K_0$ sufficiently large.  Independently for each $b\in B$, choose a
uniform $D$-element subset
$S_b\subseteq L_b^*$.
Form a graph $Q$ whose vertices are the labelled pairs $(b,u)$, with
$u\in S_b$, and where $(b,u)$ is adjacent to $(b',u')$, $b\ne b'$, if
either $u=u'$ or $uu'\in E(\Gamma)$.

Fix $b\in B$ and $u\in L_b^*$, and condition on the event $u\in S_b$.
For $b'\ne b$, the number of conflicts from $(b,u)$ into the sampled list
$S_{b'}$ is
$|S_{b'}\cap (\{u\}\cup N_\Gamma(u;L_{b'}))|$.
We shall apply \cref{fact:large-deviation} with
$\Omega_{b'}=L_{b'}^*$, and $A_{b'}=L_{b'}^*\cap(\{u\}\cup N_\Gamma(u;L_{b'}))$. The samples for different $b'$ are independent; the dependence among the
points inside each fixed-size sample is exactly the setting covered by
\cref{fact:large-deviation}. Since $S_{b'}$ is sampled from $L_{b'}^*$, and
$|L_{b'}^*|\ge |L_{b'}|/2$, the total expected number of conflicts from
$(b,u)$ into all other sampled lists is at most
$2D\ell_b(u)\le 2D\sqrt\beta$.

By \cref{fact:large-deviation}, provided $\beta$ is
small enough so that $D/2\ge e\cdot 2D\sqrt\beta$, we have
$$
        \Prb\left(\deg_Q(b,u)\ge \frac D2\ \middle|\ u\in S_b\right)
        \le
        \left(4e\sqrt\beta\right)^{D/2}
        \le
        \exp(-c'D),
$$
for some absolute constant $c'>0$.

Therefore the expected number of selected labelled vertices of degree at least
$D/2$ is at most
$$
\sum_{b\in B}\sum_{u\in L_b^*}
        \Prb(u\in S_b)
        \Prb\left(\deg_Q(b,u)\ge \frac D2\ \middle|\ u\in S_b\right)
        \le
\sum_{b\in B}\sum_{u\in L_b^*}
        \frac{D}{|L_b^*|}\exp(-c'D).
$$
This is at most
$$
        |B|D\exp(-c'D)
        \le nD\exp(-c'D).
$$
Choosing $D_0$ sufficiently large makes this expectation less than one.
Consequently, there is a choice of the sampled sets $S_b$ for which
$\Delta(Q)<\frac D2$.
Fix such a choice.

The vertex set of $Q$ is partitioned into the parts
$\{b\}\times S_b$ with $b\in B$, each of size $D$.  Since $\Delta(Q)\le D/2-1$, \cref{cor:haxell-degree}
applied with parameter $D/2$ gives an independent transversal of $Q$.  Thus
we may choose one vertex $u_b\in S_b\subseteq L_b$ for each $b\in B$, such
that the chosen vertices are distinct and no two of them are adjacent in
$\Gamma$.

Since $u_b\in L_b=C_m(b)$, all constraints between $b$ and the already
embedded vertices are satisfied.  Since $B$ is separated, there are no edges
inside $B$.  The independent transversal prevents both collisions and ambient
edges between distinct buffer images.  Hence extending $\phi_m$ by
$b\mapsto u_b$, $b\in B$, 
gives a part-respecting injective acceptable embedding of $H$ into
$\Gamma$.  This completes the proof.
\end{proof}
\section{The sparse random partite host}
\label{sec:random-host}
This section verifies the host-side hypotheses required by
\cref{thm:conflict-stable-embedding} in a sparse random partite graph.  There
are two tasks.  First, we need inheritance of lower-regularity inside ambient
neighbourhoods, uniformly over all subgraphs $G\subseteq\Gamma$.  Second, we
need the ambient conflict estimates which control induced deletions.  We begin
with the inheritance input, which is supplied by the sparse blow-up lemma
machinery of Allen, Böttcher, Hàn, Kohayakawa and Person.
\subsection{Random lower-regularity inheritance}

We use the following random inheritance result of Allen, Böttcher, Hàn,
Kohayakawa and Person, in a form obtained by taking common constants in their
one-sided and two-sided inheritance lemmas.

\begin{theorem}[Random lower-regularity inheritance {\cite[Lemmas~1.26 and~1.27]{ABHKP}}]
\label{thm:ABHKP-inheritance}
For every $0<\varepsilon',d\le1$ there are constants
$\varepsilon_{\rm ABHKP}(\varepsilon',d)>0$ and $C_{\mathrm{inh}}>0$ such
that, for every $0<\varepsilon<\varepsilon_{\rm ABHKP}(\varepsilon',d)$ and
$0<p<1$, the random graph $\Gamma=G(N,p)$ asymptotically almost surely has
the following properties.  These properties hold simultaneously for every
subgraph $G\subseteq\Gamma$.

\begin{enumerate}[label=\textup{(\roman*)}]
\item If $X,Y\subseteq V(\Gamma)$ are disjoint, $(X,Y)$ is
$(\varepsilon,d,p)$-lower-regular in $G$, and
$$
|X|\ge C_{\mathrm{inh}}\max{p^{-2},p^{-1}\log N},
\qquad
|Y|\ge C_{\mathrm{inh}}p^{-1}\log(eN/|X|),
$$
then all but at most $C_{\mathrm{inh}}p^{-1}\log(eN/|X|)$ vertices
$z\in V(\Gamma)$ are such that
$(N_\Gamma(z;X),Y)$
is $(\varepsilon',d,p)$-lower-regular in $G$.

\item If $X,Y\subseteq V(\Gamma)$ are disjoint, $(X,Y)$ is
$(\varepsilon,d,p)$-lower-regular in $G$, and
$$
|Y|\ge |X|\ge
C_{\mathrm{inh}}\max{p^{-2},p^{-1}\log N},
$$
then all but at most
$C_{\mathrm{inh}}\max{p^{-2},p^{-1}\log(eN/|X|)}$ vertices $z\in V(\Gamma)$ are such that
$(N_\Gamma(z;X),N_\Gamma(z;Y))
$
is $(\varepsilon',d,p)$-lower-regular in $G$.
\end{enumerate}
\end{theorem}

In \cite{ABHKP},  these are stated as two separate lemmas.  We have only replaced their
two sets of constants by common constants, namely the smaller admissible $\varepsilon_0$ and the larger constant $C$.  The a.a.s. events in the two
lemmas are already uniform over all subgraphs $G\subseteq\Gamma$ and all
eligible disjoint sets $X,Y\subseteq V(\Gamma)$.

\subsection{A uniform partite inheritance consequence}

We record the uniform partite consequence needed below.

\begin{corollary}[Uniform partite envelope inheritance]
\label{cor:uniform-partite-inheritance}
Fix $q\ge1$, $0<d\le1$, output and input parameters
$0<\varepsilon_{\rm out}<d$ and
$0<\varepsilon_{\rm in}<\varepsilon_{\rm ABHKP}(\varepsilon_{\rm out},d)$,
and an exceptional-set parameter $0<\zeta<1$.  There is $L>0$ such that
the following holds. Let $\Gamma$ be the random $q$-partite graph with parts
$U_1,\ldots,U_q$, each of size $M$, where every cross-edge is present
independently with probability $p$.  Let $m$ satisfy $m\ge Lp^{-2}\log M$.
Then, with probability tending to one as $M\to\infty$, the following hold
simultaneously for every subgraph $G\subseteq\Gamma$.

\begin{enumerate}[label=\textup{(\roman*)}]
\item Let $i,j,\ell\in[q]$ satisfy $i\ne j$ and $j\ne\ell$.  Let
$X\subseteq U_i$, $Y\subseteq U_j$, and $Z\subseteq U_\ell$ have sizes at
least $m$.  If
$(Y,Z)$
is $(\varepsilon_{\rm in},d,p)$-lower-regular in $G$, then all but at most
$\zeta |X|$ vertices $x\in X$ are such that
$(N_\Gamma(x;Y),Z)
$
is $(\varepsilon_{\rm out},d,p)$-lower-regular in $G$.

\item Let $i,j,\ell\in[q]$ be distinct.  Let
$X\subseteq U_i$, $Y\subseteq U_j$, and $Z\subseteq U_\ell$ have sizes at
least $m$.  If
$(Y,Z)$
is $(\varepsilon_{\rm in},d,p)$-lower-regular in $G$, then all but at most
$\zeta |X|$ vertices $x\in X$ are such that
$(N_\Gamma(x;Y),N_\Gamma(x;Z))
$
is $(\varepsilon_{\rm out},d,p)$-lower-regular in $G$.
\end{enumerate}
\end{corollary}

\begin{proof}
Apply \cref{thm:ABHKP-inheritance} with $N=qM$, input parameter
$\varepsilon_{\rm in}$, and output parameter $\varepsilon_{\rm out}$.
Since $q$ is fixed, after increasing $L$ the size condition
$m\ge Lp^{-2}\log M$ implies all size conditions appearing in
\cref{thm:ABHKP-inheritance}.  The exceptional sets in the one-sided and
two-sided conclusions have sizes at most
$C_{\mathrm{inh}}\max\{p^{-2},p^{-1}\log(e qM/m)\}
        \le \zeta m$
after increasing $L$.  Since every relevant choice set $X$ has size at
least $m$, this is at most $\zeta|X|$.

In both the one-sided and two-sided cases, the only lower-regular pair to which
we apply \cref{thm:ABHKP-inheritance} is $(Y,Z)$.  In the two-sided case,
before applying it to this pair, we order $Y$ and $Z$ so that the first
set has size at most the second; the conclusion is symmetric in the two
neighbourhoods.  The
$q$-partite model is obtained from $G(qM,p)$ by ignoring the edges inside
the parts.  The inheritance statements above involve only cross-edges between
distinct parts, so the same conclusion holds for the random partite model.  The
case $i=\ell$ in the one-sided conclusion is also allowed, because the
inherited pair $(N_\Gamma(x;Y),Z)$ then lies between $U_j$ and $U_i$,
which are distinct as $i\ne j$.
\end{proof}

\subsection{Verification of the host hypotheses}

Let $q,\Delta$ be fixed.  Let $\Gamma$ be the random $q$-partite graph
with parts $U_1,\ldots,U_q$, each of size $M$, where every cross-edge is
present independently with probability $p$.  Put
$K=p^\Delta M$.

\begin{proposition}[Random verification]
\label{prop:random-verification}
Fix $\Delta,q\ge1$, $0<d\le1$, and $0<\theta<1$.  There is a constant
$C_* = C_*(\Delta)$ such that the following holds for every prescribed
$C_{\rm good}\ge C_*$.  Let
$\boldsymbol\varepsilon=(\varepsilon_0,\ldots,\varepsilon_{2\Delta})$ and
$\boldsymbol\eta=(\eta_0,\ldots,\eta_{2\Delta-1})$ be positive parameters
such that $\varepsilon_h<\varepsilon_{\rm ABHKP}(\eta_h,d)$ and
$\varepsilon_h<d/2$ for every $0\le h<2\Delta$.  There is a constant
$L>0$, depending on the displayed parameters and on $C_{\rm good}$, such
that the following holds.  Assume
$K=p^\Delta M\ge Lp^{-2}\log M$.  Let
$\xi=\frac{L\log M}{K}$.  Then, with probability tending to one as
$M\to\infty$, the following hold simultaneously.

\begin{enumerate}[label=\textup{(\roman*)}]
\item For every subcollection of the parts of $\Gamma$, the induced subgraph on those parts is
$(\Delta,\theta,\xi,C_{\rm good})$-conflict-good at scale $M$.

\item For every subgraph $G\subseteq\Gamma$, the pair $G\subseteq\Gamma$ is
$(\Delta,\theta,\boldsymbol\varepsilon,\boldsymbol\eta,d,p)$-envelope-positively inheriting at scale $M$.
\end{enumerate}
\end{proposition}

\begin{proof}
We first prove conflict-goodness.  Choose the absolute constant $C_*$ large
enough for the edge-load estimate below.  For edge-load, it is enough to show
that, with high probability, for all
distinct parts $U_i,U_j$ and all sets
$S\subseteq U_i$, $S'\subseteq U_j$ with $|S|,|S'|\ge\theta K$, one has
\begin{equation}
        e_\Gamma(S,S')\le C_{\rm good}p|S||S'|.
        \label{eq:random-edge-load}
\end{equation}
For fixed $S,S'$, Chernoff's bound gives
$\Prb\bigl(e_\Gamma(S,S')>C_{\rm good}p|S||S'|\bigr)
        \le \exp(-cC_{\rm good}p|S||S'|)$. 
The number of choices of sets of sizes $s,t$ is at most
$\exp\{s\log(eM/s)+t\log(eM/t)\}$.  Since $s,t\ge\theta K$, we have
$$
        s\log(eM/s)+t\log(eM/t)
        \le 2(s+t)\log(eM/(\theta K))
        \le 4st(\theta K)^{-1}\log(eM).
$$
On the other hand, the Chernoff exponent is $\Omega(C_{\rm good}pst)$.  The
last displayed entropy bound is therefore dominated as soon as
$pK\gg_{\theta,C_{\rm good}}\log M$.  This follows from the stronger
assumption $K\ge Lp^{-2}\log M$, after increasing $L$, since
$pK\ge Lp^{-1}\log M\ge L\log M$.  A union bound over $s,t$ and over
pairs of parts proves \eqref{eq:random-edge-load}, and hence the edge-load
estimate for every induced subgraph on a subcollection of the parts.

For weighted overlap, no randomness is needed.  If $S_x$ and $S_y$ are in
the same part and are candidate-scale, then
$$
        \frac{|S_x\cap S_y|}{|S_x||S_y|}
        \le
        \frac1{\min\{|S_x|,|S_y|\}}
        \le
        \frac1{\theta K}.
$$
Increasing $L$, we may assume $1/(\theta K)\le \xi$.  Multiplying by the
weights $\omega_y$ and summing over $y$ gives the weighted overlap estimate.

It remains to verify envelope-positive inheritance.  The degree condition is
deterministic.  If $(S_i,S_j)$ is $(\varepsilon_h,d,p)$-lower-regular in
$G$, and more than $\varepsilon_h |S_i|$ vertices of $S_i$ have degree
less than $(d/2)p|S_j|$ into $S_j$, then those vertices together with
$S_j$ contradict lower-regularity, since $\varepsilon_h<d/2$.

For one-sided and two-sided envelope inheritance, apply
\cref{cor:uniform-partite-inheritance} for each $0\le h<2\Delta$, with
input parameter $\varepsilon_h$, output parameter $\eta_h$, exceptional
parameter $\zeta=\varepsilon_h$, and $m=\theta K$.  Since every
candidate-scale set has size at least $\theta K$, and since
$\theta K\ge \theta Lp^{-2}\log M$, increasing $L$ by a factor depending on
$\theta$ and on the finite hierarchy makes the hypotheses of
\cref{cor:uniform-partite-inheritance} valid for all candidate-scale triples.
The corollary is uniform over all subsets of the parts of size at least $m$
and over all subgraphs $G\subseteq\Gamma$, so it gives the desired
inheritance simultaneously for all relevant triples and all $h$.  This
completes the proof.
\end{proof}

Note that the random verification requires
$K=p^\Delta M\ge Lp^{-2}\log M$,
equivalently
$p^{\Delta+2}M\ge L\log M$.  In the Ramsey application we fix a small
constant $c_*>0$ and take
$p=\frac{c_*}{n\log n}$, and $ M=C_0 n^{\Delta+2}(\log n)^A$.
Then,
$p^2K
        =
        \Theta_{\Delta,r}\bigl(
        C_0c_*^{\Delta+2}(\log n)^{A-\Delta-2}
        \bigr)$.
Thus, $p^2K\ge L\log M$ after taking, for instance,
$A\ge\Delta+3$, and then choosing $C_0$ sufficiently large depending on
$c_*$ and the fixed parameters.

\section{The induced Ramsey application}
\label{sec:ramsey-application}

We now derive the induced Ramsey bound from the conflict-stable embedding
theorem and the random verification established in the previous section.

We first prove the elementary buffer lemma needed to apply
\cref{thm:conflict-stable-embedding}. To state the latter formally, we introduce the following definition. We say that a separated buffer $B\subseteq V(H)$ is $\alpha$-balanced, with respect
to the assignment $\chi$, if
$|B\cap \chi^{-1}(i)|\ge \alpha|\chi^{-1}(i)|$ for every $i\in[k]$. 

\begin{lemma}[Balanced separated buffers]
\label{lem:balanced-separated-buffer}
Fix $\Delta,k\ge1$.  There is $\alpha=\alpha(\Delta,k)>0$ such that the
following holds for all sufficiently large $n$.  Let $H$ be an $n$-vertex
graph with $\Delta(H)\le\Delta$, and let
$\chi:V(H)\to[k]$ be an equitable proper colouring.  Then $H$ contains an
$\alpha$-balanced separated buffer with respect to $\chi$.
\end{lemma}
\begin{proof}
Let $D=\Delta^2+1$.  The square $H^2$ has maximum degree at most $D-1$.
Choose $\alpha=1/(8kD)$.  We greedily construct $B$.  For each colour
class $X_i=\chi^{-1}(i)$, keep adding vertices of $X_i$ until
$|B\cap X_i|\ge \lfloor 2\alpha |X_i|\rfloor$.  During the construction, the
vertices forbidden by the current set $B$ are those within distance two of a
vertex already chosen.  Each chosen vertex forbids at most $D$ vertices in a
fixed colour class.  Before the construction terminates, we have
$|B|\le 2\alpha n$, and hence the number of forbidden vertices in any
$X_i$ is at most
$D|B|\le 2D\alpha n=\frac{n}{4k}$.
As $\chi$ is equitable, $|X_i|\ge n/(2k)$ for all sufficiently large $n$.
Thus, whenever $|B\cap X_i|<\lfloor 2\alpha |X_i|\rfloor$, some vertex of
$X_i$ remains available and can be added to $B$.  Continuing in this way,
we obtain a set $B$ with $|B\cap X_i|\ge \alpha |X_i|$ for every $i$,
for all sufficiently large $n$, and no two vertices of $B$ are at distance
at most two in $H$.  Thus $B$ is separated and $\alpha$-balanced.
\end{proof}

We shall also use the following elementary finite Ramsey selection principle.

\begin{fact}[Dense multipartite Ramsey selection]
\label{fact:dense-multipartite-ramsey-selection}
For every $k,r\ge1$ there are constants $\zeta_0>0$ and $Q$ such that the
following holds.  Let $R$ be a graph whose vertex set is partitioned into
$q\ge Q$ non-empty classes of equal size.  Suppose that at most a
$\zeta_0$-fraction of all cross-pairs of vertices of $R$ are non-edges.
Then every $r$-colouring of $E(R)$ contains a monochromatic $K_k$ whose
vertices lie in distinct classes.
\end{fact}

\begin{proof}
Choose $m$ such that every $r$-colouring of $K_m$ contains a
monochromatic $K_k$.  We shall first find a copy of $K_m$ with vertices in
$m$ distinct classes.  Choose one vertex independently and uniformly from each
class.  Since the classes have equal size, the expected number of missing edges
among the chosen vertices is at most $\zeta_0\binom q2$.  Hence there is a
transversal on $q$ vertices spanning at most $\zeta_0\binom q2$ missing
edges.  If $\zeta_0<1/(2m)$, then, for all sufficiently large $q$, Tur\'an's
theorem shows that this transversal contains a clique of order $m$.  The vertices of this clique lie in distinct classes of $R$.
Applying the ordinary multicolour Ramsey theorem to the colouring induced on
this $K_m$ gives a monochromatic $K_k$, again with vertices in distinct
classes.
\end{proof}

The next lemma is the sparse regularity input in the exact form in which it is
used below.  It is a finite simultaneous partite refinement statement.  The
point of recording it separately is that the colour-focusing argument below
uses no further regularity machinery.

If $F\subseteq\Gamma$ and $X,Y$ are disjoint vertex sets, write
$d_{F,p}(X,Y)=\frac{e_F(X,Y)}{p|X||Y|}$.
The pair $(X,Y)$ is $(\varepsilon,p)$-\emph{regular} in $F$ if, whenever
$X'\subseteq X$ and $Y'\subseteq Y$ satisfy
$|X'|\ge\varepsilon |X|$ and $|Y'|\ge\varepsilon |Y|$, we have
$|d_{F,p}(X',Y')-d_{F,p}(X,Y)|\le \varepsilon$.

\begin{lemma}[Simultaneous partite sparse regular refinement]
\label{lem:simultaneous-partite-sparse-regularity}
Fix integers $q_0,r\ge1$ and parameters
$0<\varepsilon_{\rm reg},\zeta<1/10$.  There are constants
$\mu_{\rm reg}>0$ and $\nu_{\rm reg}>0$, depending only on
$q_0,r,\varepsilon_{\rm reg},\zeta$, such that the following holds.  Let
$\Gamma$ be a $q_0$-partite graph with parts
$U_1,\ldots,U_{q_0}$, each of size $M$.  Suppose that, for every pair of
 distinct parts $U_a,U_b$ and all
$X\subseteq U_a$, $Y\subseteq U_b$ with
$|X|,|Y|\ge\nu_{\rm reg}M$, one has
$e_\Gamma(X,Y)\le 2p|X||Y|$.
Then every $r$-colouring of $E(\Gamma)$, with colour graphs
$G_1,\ldots,G_r$, admits an integer $t\ge1$ and partitions
$U_i=U_i^0\cup U_i^1\cup\cdots\cup U_i^t$, $i\in[q_0]$,
with the following properties.
\begin{enumerate}[label=\textup{(\roman*)}]
\item $|U_i^0|\le\varepsilon_{\rm reg}M$ for every $i$.
\item All non-exceptional clusters $U_i^a$, $i\in[q_0]$,
$a\in[t]$, have one common size $m$, and
$\mu_{\rm reg}M\le m\le2\mu_{\rm reg}M$.
\item For all but at most a $\zeta$-proportion of the cross-cluster pairs
$(U_i^a,U_j^b)$, with $i\ne j$, the pair is
$(\varepsilon_{\rm reg},p)$-regular in each of the colour graphs
$G_1,\ldots,G_r$.
\end{enumerate}
\end{lemma}
This is the standard sparse regularity lemma in the form of
Kohayakawa~\cite[Theorem~1]{KR}, applied with the initial partition
$U_1\cup\cdots\cup U_{q_0}$, together with its usual multicolour extension; see also the work of 
Allen, Böttcher, Hàn, Kohayakawa, and Person~\cite[Lemma~6.3]{ABHKP}
for a numbered coloured formulation.  Since the number $r$ of colour graphs
and the initial number $q_0$ of parts are fixed, the usual energy-increment
proof gives common constants depending only on
$q_0,r,\varepsilon_{\rm reg},\zeta$.

The following consequence is the colour-focusing input.  It says that, inside a
sufficiently uniform sparse partite graph, every edge-colouring contains a
monochromatic sparse lower-regular $k$-partite block.  We use it with
$k=\Delta+1$.

\begin{lemma}[Sparse colour-focusing lemma]
\label{lem:sparse-colour-focusing}
Fix $k,r\ge1$, $0<d_0\le 1/(8r)$, and $0<\varepsilon_0<d_0$.  There are constants
$q$, $\mu>0$, $\nu>0$, and $\eta>0$, depending only on
$k,r,d_0,\varepsilon_0$, such that the following holds. Let $\Gamma$ be a $q$-partite graph with parts
$U_1,\ldots,U_q$, each of size $M$.  Suppose that for every pair of
distinct parts $U_a,U_b$, and every
$X\subseteq U_a$, $Y\subseteq U_b$ with
$|X|,|Y|\ge \nu M$,
we have
\begin{equation}\label{eq: reg}
(1-\eta)p|X||Y|
        \le e_\Gamma(X,Y)
        \le (1+\eta)p|X||Y|.   
\end{equation}
Then every $r$-colouring of $E(\Gamma)$ contains distinct indices
$a_1,\ldots,a_k\in[q]$, subsets
$V_i\subseteq U_{a_i}$, $\mu M\le |V_i|\le 2\mu M$, and a colour $s\in[r]$, such that the colour-$s$ graph $G$ satisfies that
every pair $(V_i,V_j)$, $i\ne j$, is
$(\varepsilon_0,d_0,p)$-lower-regular in $G$.
\end{lemma}

\begin{proof}
Let $\zeta_0$ and $Q$ be supplied by
\cref{fact:dense-multipartite-ramsey-selection} for the present $k$ and $r$.
Choose
$0<\varepsilon_{\rm reg}\ll \zeta_0,\varepsilon_0$ and $0<\eta\ll \varepsilon_{\rm reg}, d_0$, so that
\begin{equation}
        \frac{1-2\eta}{r}-\varepsilon_{\rm reg}
        \ge d_0-\varepsilon_0 .
        \label{eq:colour-density-to-lower}
\end{equation}
This is possible because $d_0\le1/(8r)$.  Choose an integer $q\ge Q$, and
apply \cref{lem:simultaneous-partite-sparse-regularity} with parameters
$q_0=q$, $r$, $\varepsilon_{\rm reg}$, and $\zeta=\zeta_0/4$.  Let
$\mu_{\rm reg}$ and $\nu_{\rm reg}$ be the constants supplied by that lemma,
and finally choose
$\nu\le \min\{\nu_{\rm reg},\varepsilon_{\rm reg}\mu_{\rm reg}\}$.
Then the upper bound in \eqref{eq: reg} implies the upper-uniformity hypothesis of
\cref{lem:simultaneous-partite-sparse-regularity}, after decreasing $\eta$ if
necessary.

Fix an arbitrary $r$-colouring of $E(\Gamma)$, and let
$G_1,\ldots,G_r$ be the colour graphs.  Applying
\cref{lem:simultaneous-partite-sparse-regularity}, we obtain equal-sized
non-exceptional clusters
$U_i^a$, $i\in[q]$, $a\in[t]$, each of common size
$m$ with $\mu_{\rm reg}M\le m\le2\mu_{\rm reg}M$, such that all but a
$\zeta_0/4$-proportion of
cross-cluster pairs are $(\varepsilon_{\rm reg},p)$-regular in every colour.
In particular, every cluster and every subset of a cluster of relative size at
least $\varepsilon_{\rm reg}$ has size at least $\nu M$, so the large-set
uniformity hypothesis \eqref{eq: reg} applies to all sets that occur in the
regularity test.

Form a reduced graph $R$ whose vertex classes are
$\mathcal U_i=\{U_i^1,\ldots,U_i^t\}$ with $i\in[q]$.
Thus the classes of $R$ have equal size $t$.  Join two clusters
$U_i^a\in\mathcal U_i$ and $U_j^b\in\mathcal U_j$, $i\ne j$, if their
pair is $(\varepsilon_{\rm reg},p)$-regular in each colour graph.  The
regularity conclusion above shows that $R$ misses at most a
$\zeta_0/4$-proportion of all cross-cluster pairs, and hence at most a
$\zeta_0$-proportion after harmlessly increasing the error to account for the
notation of \cref{fact:dense-multipartite-ramsey-selection}.

We next colour the edges of $R$.  If $U_i^aU_j^b\in E(R)$, then, by the
lower bound in \eqref{eq: reg},
$e_\Gamma(U_i^a,U_j^b)
        \ge (1-\eta)p|U_i^a||U_j^b|$.
Since the $r$ colour graphs partition these ambient edges, some colour
$s=s(i,a,j,b)$ has
$$
        d_{G_s,p}(U_i^a,U_j^b)
        \ge \frac{1-\eta}{r}
        \ge \frac{1-2\eta}{r}.
$$
Assign one such colour to the reduced edge.

By \cref{fact:dense-multipartite-ramsey-selection}, the coloured reduced graph
contains clusters $W_1,\ldots,W_k$, lying in distinct original parts, and a
colour $s\in[r]$, such that every pair $(W_i,W_j)$ is an edge of $R$ of
colour $s$.  Hence each pair $(W_i,W_j)$ is
$(\varepsilon_{\rm reg},p)$-regular in $G_s$ and has
$p$-density at least $(1-2\eta)/r$ in $G_s$.

Set $V_i=W_i$ for $i\in[k]$, and put $\mu=\mu_{\rm reg}$.  Then the
common cluster-size bounds give
$\mu M\le |V_i|\le2\mu M$.  Finally, let
$X\subseteq V_i$ and $Y\subseteq V_j$ satisfy
$|X|\ge\varepsilon_0|V_i|$ and $|Y|\ge\varepsilon_0|V_j|$.  Since
$\varepsilon_{\rm reg}\le\varepsilon_0$, the
$(\varepsilon_{\rm reg},p)$-regularity of $(V_i,V_j)$ in $G_s$ gives
$$
        d_{G_s,p}(X,Y)
        \ge d_{G_s,p}(V_i,V_j)-\varepsilon_{\rm reg}
        \ge \frac{1-2\eta}{r}-\varepsilon_{\rm reg}
        \ge d_0-\varepsilon_0,
$$
where the last inequality is \eqref{eq:colour-density-to-lower}.  This is
exactly $(\varepsilon_0,d_0,p)$-lower-regularity.  The lemma follows.
\end{proof}

We also need the following elementary random lower-uniformity estimate.

\begin{fact}[Large-set lower-uniformity]
\label{fact:large-set-lower-uniformity}
Fix $q\ge1$, $0<\nu,\eta<1$.  There is $L>0$ such that the following
holds.  Let $\Gamma$ be the random $q$-partite graph with parts
$U_1,\ldots,U_q$, each of size $M$, and cross-edge probability $p$. If $pM\ge L\log M$, 
then with probability tending to one, for every pair of distinct parts
$U_a,U_b$ and every $X\subseteq U_a$, $Y\subseteq U_b$ with
$|X|,|Y|\ge\nu M$, we have
$$
        (1-\eta)p|X||Y|
        \le e_\Gamma(X,Y)
        \le (1+\eta)p|X||Y|.
$$
\end{fact}

\begin{proof}
For fixed $X,Y$, the random variable $e_\Gamma(X,Y)$ is binomial with mean
$p|X||Y|$.  Chernoff's bound gives
$\Prb\left(
        |e_\Gamma(X,Y)-p|X||Y||>\eta p|X||Y|
        \right)
        \le
        2\exp(-c_\eta p|X||Y|)$. 
If $|X|,|Y|\ge\nu M$, then $p|X||Y|\ge \nu^2 pM^2$.  The number of choices
of $X,Y$ is at most $4^M$.  Taking $L$ sufficiently large makes the
Chernoff exponent dominate this entropy, and a union bound over all pairs of
parts proves the claim.
\end{proof}

We now combine the random verification with the colour-focusing lemma.

\begin{proposition}[A random Ramsey-useful host]
\label{prop:random-ramsey-useful-host}
Fix $\Delta,r\ge2$.  There are constants $A,C_0,c_*>0$ and an integer $q$
such that the following holds. Let $n$ be sufficiently large, set
$k=\Delta+1$,
$p=\frac{c_*}{n\log n}$, and $M=C_0 n^{\Delta+2}(\log n)^A$.
Let $\Gamma$ be the random $q$-partite graph with
parts $U_1,\ldots,U_q$, each of size $M$, and cross-edge probability $p$. Then, the following property holds almost surely.  For every
$r$-colouring of $E(\Gamma)$, there are subsets
$V_1,\ldots,V_k$ lying in distinct parts of $\Gamma$, and a colour $s\in[r]$, such that if
$G$ is the colour-$s$ graph induced on $V_1\cup\cdots\cup V_k$, then
$G\subseteq \Gamma[V_1\cup\cdots\cup V_k]$ satisfies all host-side
hypotheses of \cref{thm:conflict-stable-embedding}, with constants depending
only on $\Delta$ and $r$, for every $n$-vertex graph $H$ with
$\Delta(H)\le\Delta$ and an equitable proper $k$-colouring.
\end{proposition}

\begin{proof}
Let $k=\Delta+1$, and choose $d_0=1/(100r)$.  Choose a fixed
conflict-good constant $C_{\rm emb}$ at least as large as the absolute
constant $C_*$ in \cref{prop:random-verification}.

Let
$\boldsymbol\varepsilon=(\varepsilon_0,\ldots,\varepsilon_{2\Delta})$,
        $\boldsymbol\eta=(\eta_0,\ldots,\eta_{2\Delta-1}),
$
and let $\theta,\xi_0,\lambda_0,K_0$ be the constants supplied by
\cref{thm:conflict-stable-embedding} with $d=d_0$, $C=C_{\rm emb}$, and
$\kappa=2$.  Inspecting the choice of constants in the proof of
\cref{thm:conflict-stable-embedding}, we may and do take the hierarchy so that,
in addition,
$\varepsilon_h<\varepsilon_{\rm ABHKP}(\eta_h,d_0)$ for every $0\le h<2\Delta$.
Choose $0<c_*\le1$ sufficiently small so that, for all sufficiently large
$n$,
$\frac{c_*\log(2n)}{\log n}
        \le \frac{\lambda_0}{4}$.
We shall take $p=c_*/(n\log n)$.  Finally choose $A\ge\Delta+3$, and then
choose $C_0$ sufficiently large, depending on all previously fixed constants,
including $c_*$.

Apply \cref{lem:sparse-colour-focusing} with this $k,r,d_0$, and with the
first regularity parameter $\varepsilon_0$.  This
gives constants $q,\mu,\nu,\eta$.  We expose the random $q$-partite graph
$\Gamma$.  By \cref{fact:large-set-lower-uniformity}, with high probability
$\Gamma$ satisfies the large-set uniformity condition (\ref{eq: reg}).  By
\cref{prop:random-verification}, applied with candidate-scale parameter
$\mu\theta$, the prescribed constant $C_{\rm emb}$, and the above
regularity hierarchies, with high probability $\Gamma$ also satisfies the
conflict-good and envelope-positive-inheritance conclusions needed below.  Here and below we
increase the constant $L$ in the definition of $\xi$ if necessary; this only
changes constants depending on $\Delta$ and $r$.

Fix an arbitrary $r$-colouring of $E(\Gamma)$.  By
\cref{lem:sparse-colour-focusing}, there are distinct original parts
$U_{a_1},\ldots,U_{a_k}$, subsets
$V_i\subseteq U_{a_i}$, $\mu M\le |V_i|\le2\mu M$, 
and a colour $s$, such that every pair $(V_i,V_j)$ is
$(\varepsilon_0,d_0,p)$-lower-regular in the colour-$s$ graph $G$.

We now pass the ambient estimates from the original parts $U_a$ to the
positive-density subparts $V_i$.  Put
$M'=\mu M$, $K'=p^\Delta M'=\mu K$, and $\xi'=\mu^{-1}\xi$.
Then $M'\le |V_i|\le2M'$.  Moreover, every
$(\Delta,\theta)$-candidate-scale set inside $V_i$, measured at scale $M'$,
is a $(\Delta,\mu\theta)$-candidate-scale set inside the original part
$U_{a_i}$, measured at scale $M$.  Consequently the edge-load estimate, the
weighted-overlap estimate, and the envelope-positive-inheritance conclusions
hold on
$\Gamma'=\Gamma[V_1\cup\cdots\cup V_k]$
at scale $M'$, with parameter $\xi'$ and with conflict-good constant
$C_{\rm emb}$.

It remains only to check the numerical hypotheses of
\cref{thm:conflict-stable-embedding}.  With
$p=c_*/(n\log n)$, and $M=C_0 n^{\Delta+2}(\log n)^A$, we have
$K'=p^\Delta M'
        =
        \Theta_{\Delta,r}\bigl(
        C_0c_*^\Delta n^2(\log n)^{A-\Delta}
        \bigr)$.
Taking $A\ge \Delta+3$ and then $C_0$ sufficiently large gives
$K'\ge K_0\log(2n)$.  The stronger condition required in
\cref{prop:random-verification} is
$K=p^\Delta M\ge Lp^{-2}\log M$,
or equivalently $p^2K\ge L\log M$.  But
$p^2K
        =
        \Theta_{\Delta,r}\bigl(
        C_0c_*^{\Delta+2}(\log n)^{A-\Delta-2}
        \bigr)$,
so, with $A\ge\Delta+3$, this also follows after increasing $C_0$.

Moreover,
$$
        p\,n\log(2n)
        =
        c_*\frac{\log(2n)}{\log n}
        \le \frac{\lambda_0}{4}
$$
for all sufficiently large $n$, by the choice of $c_*$.  The corresponding
value of $\xi' n\log(2n)$ tends to zero for every $A>\Delta+2$, and
$\xi'\le p$ for all sufficiently large $n$.  Hence, for all sufficiently
large $n$,
$0<\xi'\le\min\{\xi_0,p\}$, and $(p+\xi')n\log(2n)\le\lambda_0 $.

Since $M'\gg n$, we also have
$|\chi^{-1}(i)|\le |V_i|$ for every equitable $k$-colouring of every
$n$-vertex target graph.  All host-side hypotheses of
\cref{thm:conflict-stable-embedding} are therefore satisfied on the
$k$-partite graph $\Gamma'$, with the colour-$s$ graph $G$ as the
positive graph, using the scale $M'$, parameter $\xi'$, and conflict-good
constant $C_{\rm emb}$.
\end{proof}
We recall the Hajnal--Szemerédi theorem in the form needed below.

\begin{theorem}[Hajnal--Szemerédi~\cite{HS}]
\label{thm:hajnal-szemeredi}
Every graph $H$ with maximum degree at most $\Delta$ has an equitable
proper colouring with $\Delta+1$ colours.
\end{theorem}
We can now prove the main induced Ramsey bound.

\begin{proof}[Proof of \cref{thm:main}]
Let $k=\Delta+1$.  By \cref{thm:hajnal-szemeredi}, $H$ has an equitable proper $k$-colouring $\chi:V(H)\to[k]$.
By \cref{lem:balanced-separated-buffer}, for all sufficiently large $n$, the
graph $H$ has an $\alpha$-balanced separated buffer with respect to
$\chi$.  The bounded values of $n$ can be absorbed into the final constant.

Let $\Gamma$ be a graph satisfying the conclusion of
\cref{prop:random-ramsey-useful-host}; such a graph exists because the
probability there tends to one.  Consider an arbitrary $r$-colouring of
$E(\Gamma)$.  By \cref{prop:random-ramsey-useful-host}, there are
$k$ parts $V_1,\ldots,V_k$ and a colour $s$ such that the colour-$s$
graph $G$ and the ambient graph
$\Gamma'=\Gamma[V_1\cup\cdots\cup V_k]
$
satisfy all host-side hypotheses of \cref{thm:conflict-stable-embedding}.  The
part sizes are much larger than $n$, so
$|\chi^{-1}(i)|\le |V_i|$ for every $i$.  Together with the separated
buffer chosen above, all hypotheses of the embedding theorem are satisfied.
Applying that theorem gives an acceptable embedding 
$\phi:V(H)\to V(\Gamma')$. Thus the positive edges of the copy all have colour $s$.  Since $\Gamma'$ is
the induced subgraph of $\Gamma$ on $V_1\cup\cdots\cup V_k$, the copy is
also induced in $\Gamma$. Hence every $r$-colouring of $E(\Gamma)$ contains a monochromatic copy of
$H$ which is induced in $\Gamma$.  
Finally,
$|V(\Gamma)|=qM
        \le C n^{\Delta+2}(\log n)^A,
$
where $q$ depends only on $\Delta$ and $r$.  This proves the claimed
bound.
\end{proof}

\end{document}